%% file: ColCov.tex
\documentclass[11pt]{amsart}

\usepackage{amsmath,amssymb,amscd,amsthm}

\usepackage[dvips]{graphicx}
\usepackage{epsfig}
\usepackage{graphics}
\usepackage{color}
\graphicspath{{Fig/}}

\newtheorem{thm}{Theorem}
\newtheorem{cor}{Corollary}[section]
\newtheorem{dfn}[cor]{Definition}

\theoremstyle{definition}
\newtheorem{exl}[cor]{Example}
\newtheorem{rem}[cor]{Remark}

\def\Z{{\mathbb Z}}

\def\Sph{{\mathbb S}}

\def\Sym{\mathrm{Sym}}
\def\Aut{\mathrm{Aut}}
\def\So{\Sigma_{odd}}
\def\wS{\widetilde{\Sigma}}
\def\mod{\mathrm{mod}\ }
\def\Vert{\mathrm{Vert}}
\def\dev{\mathrm{dev}}

\title{Color or cover}
\author{Ivan Izmestiev}
\date{\today}
\thanks{Supported by the European Research Council under the European Union's Seventh Framework Programme (FP7/2007-2013)/\allowbreak ERC Grant agreement no.~247029-SDModels.}
\address{Institut f\"ur Mathematik \\
Freie Universit\"at Berlin \\
Arnimallee 2 \\
D-14195 Berlin \\
 GERMANY}
\curraddr{University of Fribourg Department of Mathematics\\ Chemin du Mus\'ee 23 \\ CH-1700 Fribourg P\'erolles \\ SWITZERLAND}
\email{ivan.izmestiev@unifr.ch}

\begin{document}

\begin{abstract}
If all but two vertices of a triangulated sphere have degrees divisible by $k$, then the exceptional vertices are not adjacent. This theorem is proved for $k=2$ with the help of the coloring monodromy. For $k = 3, 4, 5$ colorings by the vertices of platonic solids have to be used.
With a coloring monodromy one can associate a branched cover. This generalizes to a space of germs between two triangulated surfaces.

We also discuss relations with Belyi surfaces and with cone-metrics of constant curvature.
\end{abstract}

\maketitle

\section{Introduction}
One of the results of this article is a new proof of the following theorem of Steve Fisk.
\begin{thm}[Fisk \cite{Fisk73,Fisk78}]
\label{thm:OddNeighb}
If a triangulation of the $2$-sphere has exactly two vertices of odd degree, then these vertices are not adjacent.
\end{thm}

This statement looks quite surprising.
Indeed, from a double counting argument we know that every graph has an even number of vertices of odd degree. Of course, there exist abstract graphs with exactly two odd vertices, which are adjacent.
There also exist triangulated spheres with two non-adjacent vertices of odd degree, or with four pairwise adjacent odd vertices, see Figure \ref{fig:OddVert}. But nothing can bring together two single odd vertices.

\begin{figure}[ht]
\centering
\includegraphics{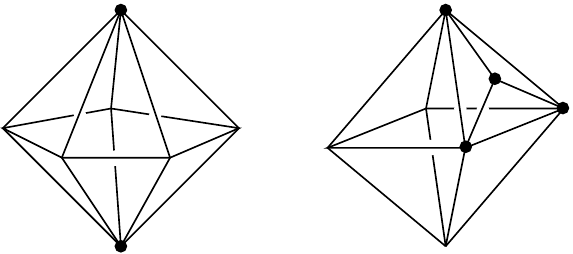}
\caption{Some possible configurations of vertices of odd degree.}
\label{fig:OddVert}
\end{figure}

The proof of Theorem \ref{thm:OddNeighb} given in \cite{Fisk78} is a straightforward combination of a lemma from \cite{Fisk73} with the four-color theorem (whose proof appeared in the time between \cite{Fisk73} and \cite{Fisk78}). The proofs presented in this article are much simpler; one of them uses the ``even obstruction map'' introduced by Fisk in \cite{Fisk77}. We prefer to use a different term ``coloring monodromy''.

By using a generalization (also due to Fisk \cite{Fisk77}) of the coloring monodromy, we prove the following generalization of Theorem \ref{thm:OddNeighb}.

\begin{thm}
\label{thm:kVert}
Let $k$ be a positive integer. If a triangulation of the $2$-sphere has exactly two vertices whose degrees are not divisible by $k$, then these vertices are not adjacent.
\end{thm}
The case $k \ge 6$ is trivial: the sum of the vertex degrees of a triangulated sphere with $n$ vertices equals $6n-12$, and having all but two vertex degrees multiples of $k$ would lead to a larger sum.
Thus only the cases $k = 2, 3, 4, 5$ are of interest.

Theorem \ref{thm:OddNeighb} implies the non-planarity of every graph on $n$ vertices with $3n-6$ edges and with exactly two vertices of odd degree, which are adjacent. However, the simplest such graphs are just vertex-splittings of $K_5$.

The article is organized as follows. In Section \ref{sec:ColPol} an elementary proof of Theorem \ref{thm:OddNeighb} is given. It uses vertex colorings of even triangulations of simply-connected surfaces. In Section \ref{sec:3Col} we define the \emph{coloring monodromy} which provides a more formal framework, suitable for generalizations. One of these is the \emph{platonic monodromy}, which leads to a proof of Theorem \ref{thm:kVert} in Section \ref{sec:PlatMon}.

Section \ref{sec:BrCov} contains further developments of the notions introduced in Section \ref{sec:ColMon}. We discuss the branched covers defined by the coloring monodromies, mention a relation to the Belyi surfaces, and give a geometric version of the proof of Theorem \ref{thm:kVert}, based this time on the developing map of a spherical cone-surface on the sphere.

\section{The coloring monodromy}
\label{sec:ColMon}
\subsection{Even triangulations of polygons}
\label{sec:ColPol}
Assume we succeeded to find a triangulation of the sphere with exactly two odd degree vertices, which are neighbors. Remove the edge joining the odd vertices, together with the two triangles on both sides of it. We obtain a triangulation of a sphere with a quadrangular hole, where all vertices, interior as well as those on the boundary, are of even degree. By stretching it to the plane, we obtain a triangulation of the square.
Thus it suffices to show that there is no triangulation of the square with all vertices of even degree.

Let us study a more general problem: {\em for which $n$ can an $n$-gon be triangulated with all vertices of even degrees?} One tries in vain to do this with the quadrilateral and the pentagon, finds an even triangulation for a hexagon, and comes up with a construction that produces an even triangulation of an $(n+3)$-gon from one for an $n$-gon, see Figure \ref{fig:NgonEven}.

\begin{figure}[ht]
\centering
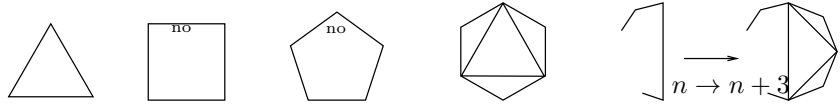
\caption{Even triangulations of $n$-gons.}
\label{fig:NgonEven}
\end{figure}

\begin{thm}
\label{thm:ColorPolygon}
There is a triangulation of an $n$-gon (vertices in the interior allowed, vertices on the sides forbidden) with all vertices of even degree if and only if $n$ is divisible by $3$.
\end{thm}
\begin{proof}
It is a part of folklore that the vertices of an even triangulation of an $n$-gon can be properly colored in three colors. Indeed, a coloring of a triangle induces colorings of all adjacent triangles. Thus we have little choice but to color one of the triangles and extend this coloring along all sequences of triangles where consecutive triangles share an edge. Extensions along different paths ending at the same triangle will agree because the degrees of all interior vertices are even: sliding a path over a vertex of even degree does not change the colors at its terminal triangle.

\begin{figure}[ht]
\centering
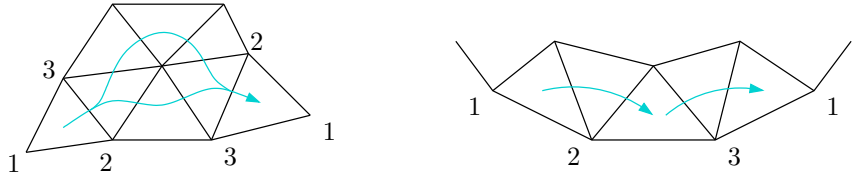
\caption{Coloring an even triangulation of an $n$-gon.}
\end{figure}

Now look at the colors assigned to the corners of the $n$-gon. Since all boundary vertices have even degrees, the colors on the boundary repeat cyclically $1 \to 2 \to 3 \to 1 \to \cdots$. It follows that $n$ is divisible by $3$.

To triangulate $n$-gons for $n$ divisible by $3$, use the inductive construction shown on Figure \ref{fig:NgonEven}.
\end{proof}

\subsection{Coloring in three colors}
\label{sec:3Col}
Let $\Sigma$ be a triangulation of a surface, possibly with boundary.
Try to color in three colors the vertices of $\Sigma$ so that neighbors have different colors. As in the proof of Theorem \ref{thm:ColorPolygon}, we can start by coloring the vertices of an arbitrary triangle and extend the coloring along every sequence of triangles where consecutive triangles share an edge. In general, the coloring of the last triangle depends on the choice of a sequence. Two homotopic sequences induce the same coloring if the homotopy doesn't pass through vertices of odd degree.

In particular, we can consider only closed paths starting and ending at a triangle $\sigma_0$ and look how the extension along these paths permutes the colors.

\begin{dfn}
\label{dfn:ColMon}
Let $\Sigma$ be a triangulation of a surface, possibly with boundary, and let $a_1, \ldots, a_n$ be all interior vertices of odd degree. Pick a triangle $\sigma$ and fix a \emph{coloring} of its vertices, that is a bijection between its vertex set and the set $\{1,2,3\}$. The group homomorphism
\[
\pi_1(|\Sigma| \setminus \{a_1, \ldots, a_n\}, \sigma) \to \Sym_3
\]
that maps every path to the associated recoloring of the vertices of $\sigma_0$ is called the \emph{coloring monodromy}.
\end{dfn}

More generally, the coloring monodromy of a triangulated $d$-dimensional manifold is a homomorphism
\[
\pi_1(|\Sigma| \setminus |\Sigma_{odd}|), \sigma) \to \Sym_{d+1}
\]
where $\Sigma_{odd}$ is the odd locus of the triangulation, see Section \ref{sec:HighDim}.
Clearly, a triangulation is vertex-colorable if and only if its coloring monodromy is trivial. The coloring monodromy was introduced in \cite{Fisk77} (the even obstruction map) and rediscovered in \cite{J02} (the group of projectivities).

\begin{proof}[Second proof of Theorem \ref{thm:OddNeighb}]
Assume we have a triangulation of the sphere with two adjacent odd degree vertices $a$ and $b$. Let $\sigma$ be one of the triangles containing $a$ and $b$; color the vertices $a$ and $b$ with the colors $1$ and $2$, respectively. Then the path going around $a$ exchanges the colors $2$ and $3$, and the path around $b$ exchanges $1$ and $3$.

It follows that the image of the coloring monodromy
\[
\pi_1(\Sph^2 \setminus \{a, b\}) \to \Sym_3
\]
is the whole group $\Sym_3$. On the other hand, $\pi_1(\Sph^2 \setminus \{a, b\}) \cong \Z$. Since there is no epimorphism from the group $\Z$ onto $\Sym_3$, there is no triangulation with the assumed properties.
\end{proof}

Note that the torus and the projective plane can be triangulated with two adjacent odd vertices, see Figure \ref{fig:TorProjOdd}. The above proof of Theorem \ref{thm:OddNeighb} gives a reason why this can happen: the fundamental groups of doubly punctured torus or projective plane is large enough to allow an epimorphism onto $\Sym_3$.

\begin{figure}[ht]
\centering
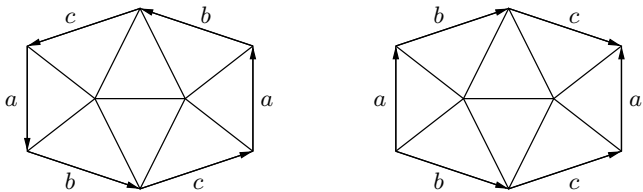
\caption{Triangulations of the projective plane and of the torus with two odd adjacent vertices.}
\label{fig:TorProjOdd}
\end{figure}

\subsection{Colorings by the vertices of platonic solids}
\label{sec:PlatMon}
The above arguments were based on the observation that a simply connected triangulated surface with all vertices of even degrees has a proper vertex coloring in three colors. For Theorem \ref{thm:kVert} we need a new notion of a proper coloring adapted to the case when all vertex degrees are divisible by $k$, for $k = 3, 4, 5$.

Let $\Sigma'$ be the boundary complex of the tetrahedron, octahedron, or icosahedron for $k = 3$, $4$, or $5$, respectively. The vertices of $\Sigma'$ will be our colors; that is, a coloring of $\Sigma$ associates to each vertex of $\Sigma$ a vertex of $\Sigma'$. A coloring of $\Sigma$ is called \emph{proper}, if the following two conditions are satisfied.
\begin{enumerate}
\item
The vertices of any triangle of $\Sigma$ are bijectively colored by the vertices of some triangle of $\Sigma'$.
\item
The vertices of any two adjacent triangles of $\Sigma$ are bijectively colored by the vertices of some two adjacent triangles of $\Sigma'$.
\end{enumerate}

The properties (1) and (2) ensure that a coloring of a base triangle $\sigma$ can be uniquely extended along a path of adjacent triangles. If the surface is simply-connected and all of its interior vertices have degrees divisible by $k$, then extensions along different paths agree, and we obtain a proper coloring of the whole surface.

In general, going along a closed path results in a recoloring of the vertices of the initial triangle: starting with $\phi_0 \colon \sigma \to \sigma'$ we end up with $\phi_1 \colon \sigma \to \sigma''$. The composition $\phi_1 \circ \phi_0^{-1} \colon \sigma' \to \sigma''$ has a unique extension to an automorphism of the regular polyhedron $\Sigma'$. (The automorphism group of $\Sigma'$ is flag-transitive.)

Thus, similarly to Definition \ref{dfn:ColMon} we have the following construction, which is due to Fisk \cite[Section I.5]{Fisk77}.

\begin{dfn}
Let $k \in \{3, 4, 5\}$, and let $\Sigma'$ be the boundary complex of the tetrahedron, octahedron, or icosahedron, respectively.

Let $\Sigma$ be a triangulated surface, and let $a_1, \ldots, a_n$ be all its interior vertices whose degrees are not divisible by $k$. Pick a triangle $\sigma$ of $\Sigma$ and fix a \emph{coloring} of its vertices, that is a bijection between its vertex set and the vertex set of some triangle of $\Sigma'$. The group homomorphism
\[
\pi_1(M \setminus \{a_1, \ldots, a_n\}, \sigma) \to \Aut(\Sigma')
\]
that maps every path to the automorphism of $\Sigma'$ associated with the recoloring of the vertices of $\Delta_0$ is called the \emph{platonic monodromy}.
\end{dfn}

Intuitively, the platonic monodromy arises from ``rolling'' the platonic solid $\Sigma'$ over the simplicial surface $\Sigma$.

\begin{proof}[Proof of Theorem \ref{thm:kVert}]
Assume we have a triangulation of the sphere with two adjacent vertices $a$ and $b$ whose degrees are not divisible by $k$, the degrees of all other vertices being multiples of $k$. Let $\sigma$ be one of the triangles containing $a$ and $b$. Then the two closed paths, one going around $a$ and the other going around $b$ generate two non-commuting automorphisms of $\Sigma'$. Hence the image of the $\Sigma'$-coloring monodromy is a non-abelian group. On the other hand, since $\pi_1(\Sph^2 \setminus \{a,b\}) \cong \Z$, the image must be a cyclic group. The contradiction shows that such a triangulation doesn't exist.
\end{proof}

\begin{rem}
Theorem \ref{thm:kVert} can be strengthened: for $k > 2$ the two exceptional vertices not only can't be adjacent, but also can't lie across an edge. More generally, rolling $\Sigma'$ from one exceptional vertex to the other must color the second vertex in the color of the first.
\end{rem}

\begin{rem}
The case $k=3$ can be described as coloring vertices of $\Sigma$ in $4$ colors so that both the neighbors and the vertices lying across an edge have different colors. For $k=4$ one colors the vertices in $6$ colors imitating the numbers on the dice: the colors at the vertices lying across an edge must add up to $7$. For $k=5$ a special arrangement of $12$ colors is used, which is not so easy to describe explicitely.
\end{rem}

\section{Branched covers}
\label{sec:BrCov}
\subsection{The minimal vertex-colored branched cover}
\label{sec:Unfolding}
The coloring monodromy (Definition \ref{dfn:ColMon}) not only detects whether a triangulation is vertex-colorable, but also allows to associate with every non-colorable triangulation a colorable branched cover. Namely, assume that some closed path forces us to recolor the vertices of a triangle. Then, instead of changing the colors, start a new layer of triangles. See Figure \ref{fig:BrCov} for an illustration.

\begin{figure}[ht]
\centering
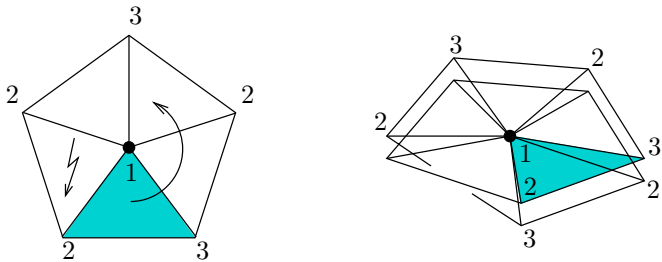
\caption{Non-trivial coloring monodromy generates a branched cover.}
\label{fig:BrCov}
\end{figure}

Applying this procedure to all paths produces a branched cover over the triangulated surface, ramified over the vertices of odd degree.
It was introduced in \cite{IJ03}.

\begin{dfn}
\label{dfn:Unfold}
Let $\Sigma$ be a triangulation of a surface $M$. Consider the set of colored triangles of $\Sigma$:
\[
\{(\sigma, \phi) \mid \sigma \text{ a triangle of }\Sigma,\ \phi \colon \Vert(\sigma) \to \{1,2,3\} \text{ a bijection}\}
\]
If two triangles $\sigma_1$ and $\sigma_2$ share an edge $\tau$, and the colorings $\phi_1$ and $\phi_2$ agree on $\tau$, then glue $(\sigma_1, \phi_1)$ to $(\sigma_2, \phi_2)$ along $\tau$.
The resulting simplicial complex $\widetilde{\Sigma}$ is called the \emph{unfolding} of $\Sigma$.
\end{dfn}

There is a canonical projection $\wS \to \Sigma$.
If the unfolding is disconnected, then all of its components are isomorphic to each other.

\begin{exl}
The unfolding of the boundary of the simplex is a triangulation of the torus. The projection $\wS \to \Sigma$ is a six-fold branched cover, where every vertex has three preimages of index $2$. See \cite[Figure 3]{IJ03}.
\end{exl}

If all vertices of $\Sigma$ have even degree, then $\wS \to \Sigma$ is a usual (non-branched) covering, which is non-trivial if and only if there is a homotopically non-trivial loop that permutes the colors.

\begin{exl}
In the $7$-vertex triangulation of the torus all vertices have degree $6$. Generators of the fundamental group recolor the vertices of a triangle in a cyclic way, see Figure \ref{fig:7Torus}. Therefore the unfolding is a vertex-colorable triangulation on 21 vertices triply covering the $7$-vertex triangulation.
\end{exl}

\begin{figure}[ht]
\centering
\includegraphics{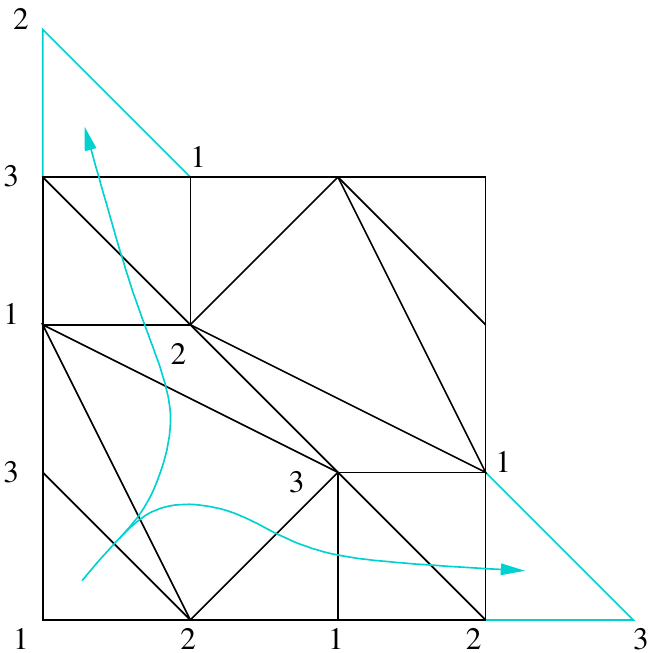}
\caption{The coloring monodromy of the $7$-vertex torus triangulation.}
\label{fig:7Torus}
\end{figure}


\begin{thm}
Every even triangulation of the torus is either vertex-colorable or can be cut along a simple closed curve so that to become vertex-colorable.
\end{thm}
\begin{proof}
It suffices to find a primitive element of $\Z^2 \cong \pi_1(T)$ with the trivial coloring monodromy. There is at least one such among $e_1, e_2, e_1 \pm e_2$, where $(e_1, e_2)$ is a basis of $\Z^2$.
\end{proof}


\subsection{Belyi surfaces}
A \emph{Belyi function} is a holomorphic map from a compact Riemann surface $M$ to $\mathbb{C}P^1$ ramified over three points $0, 1, \infty$. Subdivide $\mathbb{C}P^1$ in two triangles with vertices $0, 1, \infty$ and color one of the triangles white and the other black. This induces a triangulation of $M$ with the vertices colored with $0$, $1$, and $\infty$, and the triangles colored white and black. For more details see \cite{LZ04}.

Note that the existence of a vertex-coloring does not imply the existence of a face-coloring, and vice versa, see Figure \ref{fig:VertFaceCol}.

\begin{figure}[ht]
\centering
\includegraphics{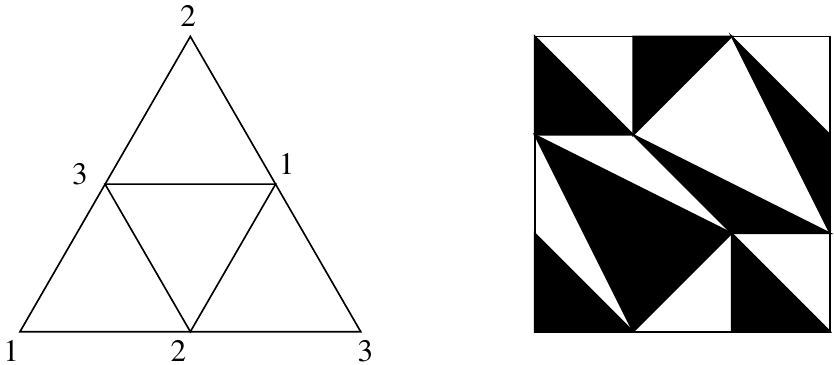}
\caption{Vertex-colored but not face-colored projective plane; face-colored but not vertex-colored torus.}
\label{fig:VertFaceCol}
\end{figure}

\begin{thm}
A vertex-colorable triangulation is face-colorable if and only if the underlying surface is oriented.

An even triangulation of an orientable surface is face-colorable if and only if the image of the coloring monodromy is either trivial or is generated by a $3$-cycle.
\end{thm}
\begin{proof}
A vertex- and face-colored triangulation can be oriented by choosing the $(123)$-orientation for every white and the $(132)$-orientation for every black triangle. Vice versa, if a vertex-colored triangulation is oriented, then we can color every triangle with orientation $(123)$ white and every triangle with orientation $(132)$ black.

For the second part, color the faces of the unfolding $\wS$. The assumption on the monodromy implies that the faces of $\wS$ with the same image in $\Sigma$ have the same color. This produces a face-coloring of $\Sigma$.
\end{proof}

Similarly to Definition \ref{dfn:Unfold}, one defines a minimal face- and vertex-colored branched cover of a given triangulation. For this, take the set of vertex- and face-colored triangles $(\Delta, \phi, \epsilon)$, where $\epsilon \in \{w,b\}$, and glue a pair of adjacent colored triangles along their common edge if their vertex colors coincide and the face colors are opposite.


\subsection{The space of germs}
The platonic monodromy from Section \ref{sec:PlatMon} generates a branched cover over a triangulated surface, similarly to Section \ref{sec:Unfolding}. A closer look at the definition of a proper platonic coloring shows that $\Sigma$ and $\Sigma'$ can be interchanged. In particular, $\Sigma'$ might be any triangulated surface, not necessarily the boundary of a platonic solid. The resulting simplicial complex covers $\Sigma$, if $\Sigma'$ has no boundary, and vice versa.

\begin{dfn}
Let $\Sigma$ and $\Sigma'$ be two triangulated surfaces without boundary. Consider the set of all triples $\Delta = (\sigma, \sigma', \phi)$, where $\sigma \in \Sigma$ and $\sigma' \in \Sigma'$ are triangles, and $\phi \colon \sigma \to \sigma'$ is a bijection between their vertex sets. (In particular, $\Delta$ can be viewed as a triangle.)

Assume that $\Delta_1 = (\sigma_1, \sigma'_1, \phi_1)$ and $\Delta_2 = (\sigma_2, \sigma'_2, \phi_2)$ are such that the triangles of the same surface share an edge, and the restrictions of $\phi_1$ and $\phi_2$ agree:
\[
\sigma_1 \cap \sigma_2 = \tau, \quad \sigma'_1 \cap \sigma'_2 = \tau', \quad \phi_1|_{\tau} = \phi_2|_{\tau}
\]
Identify the triangles $\Delta_1$ and $\Delta_2$ along their common edge according to the map $\phi_1$ (or $\phi_2$). The resulting simplicial complex is called the \emph{space of germs} between $\Sigma$ and $\Sigma'$.
\end{dfn}

The name ``space of germs'' is chosen by analogy with germs of analytic functions: an affine isomorphism $\phi \colon \sigma \to \sigma'$ has a unique continuation across every edge of $\sigma$.

Note that the space $G(\Sigma, \Sigma')$ can be disconnected.

\begin{exl}
If $\Sigma'$ is isomorphic to $\Sigma$, then $G(\Sigma, \Sigma')$ contains a component isomorphic to $\Sigma$. Other components will be more complicated, unless $\Sigma$ has a flag-transitive symmetry group.
\end{exl}

\begin{exl}
Let $\Sigma$ and $\Sigma'$ be the boundaries of the tetrahedron and of the octahedron, respectively. Then $G(\Sigma, \Sigma')$ consists of two isomorphic components: a tetrahedron ``rolls transitively'' on the outside and on the inside of the octahedron. Each of the components is an equivelar surface of vertex degree $12$. It has $96$ faces, hence $144$ edges and $24$ vertices, and therefore genus $13$.
\end{exl}

\begin{thm}
Each component of the space of germs $G(\Sigma, \Sigma')$ is a simplicial branched cover over both $\Sigma$ and $\Sigma'$. Any triangulated surface that branched covers $\Sigma$ and $\Sigma'$ covers also one of the components of $G(\Sigma, \Sigma')$.
\end{thm}
\begin{proof}
The map $G(\Sigma, \Sigma') \to \Sigma$ defined by $(\sigma, \sigma', \phi) \mapsto \sigma$ is a branched cover, which is immediate from the definition.

Given two branched covers $f \colon Z \to \Sigma$ and $f' \colon Z \to \Sigma'$, put
\[
F \colon Z \to G(\Sigma, \Sigma'), \quad \zeta \mapsto (f(\zeta), f'(\zeta), f' \circ f^{-1})
\]
\end{proof}

\begin{rem}
In \cite[Section I.3]{Fisk77}, Fisk defined the ``product'' of two pure simplicial complexes, which is a quotient of the space of germs under the identification of some lower-dimensional simplices.

Stephen Wilson defined in \cite{Wil94} the \emph{parallel product} of maps on surfaces. In the case when the maps are triangulations, the parallel product becomes our space of germs.

In higher dimensions a similar operation of \emph{mixing} of two abstract regular polytopes was introduced by Peter McMullen and Egon Schulte in \cite{McMSch}.
\end{rem}

\subsection{Geometric aspect}
The proof of Theorem \ref{thm:kVert} from Section \ref{sec:PlatMon} can be given a geometric flavor.

\begin{proof}[Second proof of Theorem \ref{thm:kVert}]
Assume we have a triangulation of the sphere with all vertex degrees divisible by $k$ except two adjacent vertices $a$ and $b$. Make every triangle spherical with side lengths $\frac{2\pi}k$. (For $k=2$ such a triangle covers a hemisphere, with vertices equally spaced on a big circle.) The assumption on the vertex degrees implies that the total angles around all vertices except $a$ and $b$ are multiples of $2\pi$. Cut the sphere along the edge joining $a$ and $b$ and remove all interior vertices. The result is a surface $\overline{M}$ (non-compact and with boundary) equipped with a spherical metric. The holonomy representation \cite[Chapter 1.4]{CHK} of this metric structure is trivial, hence there exists a developing map (local isometry) $\dev \colon \bar M \to \Sph^2$.

\begin{figure}[ht]
\centering
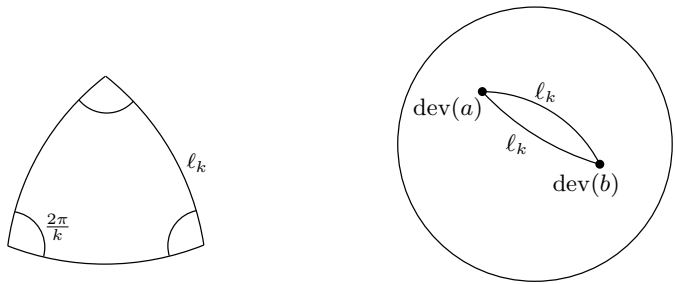
\caption{Developing a surface with a spherical metric onto the sphere.}
\label{fig:ConjPoints}
\end{figure}

Look at the images of the vertices $a$ and $b$ under the developing map. Each side of the slit made along the edge $ab$ becomes a geodesic arc with the endpoints $\dev(a)$ and $\dev(b)$. These geodesic arcs are different because the degrees of $a$ and $b$ are not divisible by $k$. The length of each arc equals the side length $a_k$ of the equilateral spherical triangle with the angle $\frac{2\pi}k$. We have $a_k \le \frac{2\pi}3$; on the other hand two points on the sphere can be joined by two different geodesics only if they are antipodal. This contradiction shows that a triangulation with the given properties does not exist.
\end{proof}

The next theorem generalizes the well-known fact that there is no convex polyhedron with $12$ pentagonal faces and one hexagonal face.
\begin{thm}
\label{thm:Penta}
There is no triangulation of the sphere with the degrees of all vertices except one divisible by $5$.
\end{thm}
\begin{proof}
Assume that all vertex degrees except one are divisible by $5$. Make every triangle spherical equilateral with the angle $\frac{2\pi}5$. On one hand, the holonomy around the exceptional vertex is non-trivial; on the other hand it is trivial since it is the product of the holonomies around all other points.
\end{proof}
For example, there is no polyhedron made of one triangle and nine pentagons, and no polyhedron made of $13$ pentagons and one heptagon.

Again, instead of $5$ in Theorem \ref{thm:Penta} one can take any number $k$; but for $k \ne 5$ this is trivial because the sum of all vertex degrees equals $6n-12$.

\begin{rem}
Jendrol' and Jucovi{\v{c}} proved in \cite{JJ72} that there is no triangulation of the torus with all vertices of degree 6 except two of degrees $5$ and $7$. In \cite{IKRSS13} a new proof is given; it uses the holonomy of the geometric structure arising when all triangles are viewed as euclidean equilateral ones.
\end{rem}

\subsection{Odd subcomplex in higher dimensions}
\label{sec:HighDim}
Let $\Sigma$ be a $d$-dimensional triangulated manifold. Define the \emph{odd subcomplex} $\So$ of $\Sigma$ as the union of all $(d-2)$-simplices incident to an odd number of $d$-simplices. An \emph{even triangulation} is one whose odd subcomplex is empty.
Even triangulations of $d$-dimensional manifolds are studied in \cite{RubTil14} in the context of geometric topology.

The unfolding map $\wS \to \Sigma$ from Section \ref{sec:Unfolding} can be defined for a triangulated manifold of any dimension; it is ramified over the odd subcomplex.

While the unfolding of every triangulated surface is again a surface, this is no more true for manifolds of dimension $d \ge 3$. For example, the unfolding of the double tetrahedron contains points with neighborhoods homeomorphic to the suspension over the torus. For $\wS$ to be a manifold, $\So$ must be a submanifold of $\Sigma$. In particular, if $\dim \Sigma = 3$, then $\So$ must be a knot or a link. This case, with $\Sigma$ a triangulated $3$-sphere, was studied in \cite{IJ03}.

The following are some restrictions on the odd subcomplex.
\begin{thm}
\begin{enumerate}
\item For every pure simplicial complex $\Sigma$, the fundamental cycle of its odd subcomplex $\So$ represents a zero element in $H_{d-2}(M;\Z_2)$.

\item If $d \equiv 0$ or $3$ $(\mod 4)$, then $\sharp(\So)$ is even; if $d \equiv 1$ or $2$ $(\mod 4)$, then $\sharp(\So)$ has the same parity as $\sharp(\Sigma)$. Here $\sharp(\So)$ and $\sharp(\Sigma)$ denote the number of $(d-2)$- and $d$-faces in $\So$ and $\Sigma$, respectively.

\item Let $\tau \subset \So$, $\dim \tau = d-3$. Assume that there are only two odd faces $\sigma_1, \sigma_2 \supset \tau$. Then $\sigma_1$ and $\sigma_2$ are not contained in the same $(d-1)$-simplex.
\end{enumerate}
\end{thm}
\begin{proof}
The complex $\So$ is the $\Z_2$-boundary of the $(d-1)$-skeleton of $\Sigma$.

For the second part, count the incidences between $d$- and $(d-2)$-faces of $\Sigma$. Every $d$-simplex has $\frac{d(d+1)}2$ faces of dimension $d-2$. This number is even if and only if $d \equiv 0$ or $3$ $(\mod 4)$. Therefore
\begin{multline*}
\sharp(\So) \equiv \sum_{\dim\sigma = d-2} i_\sigma \\
= \frac{d(d+1)}2 \sharp(\Sigma) \equiv
\begin{cases}
0\ (\mod 2), &\text{if } d \equiv 0 \text{ or } 3\ (\mod 4)\\
\sharp(\Sigma)\ (\mod 2), &\text{if } d \equiv 1 \text{ or } 2\ (\mod 4)
\end{cases}
\end{multline*}

The third part of the theorem follows by considering the link of $\tau$. It is a triangulated $2$-sphere whose odd subcomplex consists of two points $\sigma_1 \setminus \tau$ and $\sigma_2 \setminus \tau$. By Theorem \ref{thm:OddNeighb}, these points are not adjacent, which means that $\sigma_1$ and $\sigma_2$ don't lie in the same $(d-1)$-simplex.
\end{proof}

Every codimension $2$ submanifold $N \subset M$ which is a boundary $\mod 2$ can be realized as the odd subcomplex of some triangulation of $M$, see \cite[Proposition 5.1.3]{IJ03}.

\section{Acknowledgments}
The author thanks Barry Monson for pointing out the references \cite{McMSch} and~\cite{Wil94}, and an unknown referee for drawing my attention to the articles of Steve Fisk.

The article was written during the author's stay at the Pennsylvania State University as a Shapiro visitor.

\end{document}

%% file: NgonEven.pdf_t
\begin{picture}(0,0)%
\includegraphics{NgonEven.pdf}%
\end{picture}%
\setlength{\unitlength}{4144sp}%
\begingroup\makeatletter\ifx\SetFigFont\undefined%
\gdef\SetFigFont#1#2#3#4#5{%
  \reset@font\fontsize{#1}{#2pt}%
  \fontfamily{#3}\fontseries{#4}\fontshape{#5}%
  \selectfont}%
\fi\endgroup%
\begin{picture}(4980,829)(-2357,-746)
\put(-1373,-325){\makebox(0,0)[lb]{\smash{{\SetFigFont{6}{7.2}{\rmdefault}{\mddefault}{\updefault}{\color[rgb]{0,0,0}no}%
}}}}
\put(-448,-331){\makebox(0,0)[lb]{\smash{{\SetFigFont{6}{7.2}{\rmdefault}{\mddefault}{\updefault}{\color[rgb]{0,0,0}no}%
}}}}
\put(1621,-691){\makebox(0,0)[lb]{\smash{{\SetFigFont{10}{12.0}{\rmdefault}{\mddefault}{\updefault}{\color[rgb]{0,0,0}$n \to n+3$}%
}}}}
\end{picture}%

%% file: EvenColor.pdf_t
\begin{picture}(0,0)%
\includegraphics{EvenColor.pdf}%
\end{picture}%
\setlength{\unitlength}{4144sp}%
\begingroup\makeatletter\ifx\SetFigFont\undefined%
\gdef\SetFigFont#1#2#3#4#5{%
  \reset@font\fontsize{#1}{#2pt}%
  \fontfamily{#3}\fontseries{#4}\fontshape{#5}%
  \selectfont}%
\fi\endgroup%
\begin{picture}(5077,1072)(-2181,-581)
\put(-2166,-526){\makebox(0,0)[lb]{\smash{{\SetFigFont{10}{12.0}{\rmdefault}{\mddefault}{\updefault}{\color[rgb]{0,0,0}$1$}%
}}}}
\put(-869,-482){\makebox(0,0)[lb]{\smash{{\SetFigFont{10}{12.0}{\rmdefault}{\mddefault}{\updefault}{\color[rgb]{0,0,0}$3$}%
}}}}
\put(-278,-334){\makebox(0,0)[lb]{\smash{{\SetFigFont{10}{12.0}{\rmdefault}{\mddefault}{\updefault}{\color[rgb]{0,0,0}$1$}%
}}}}
\put(-1614,-522){\makebox(0,0)[lb]{\smash{{\SetFigFont{10}{12.0}{\rmdefault}{\mddefault}{\updefault}{\color[rgb]{0,0,0}$2$}%
}}}}
\put(-1955, 39){\makebox(0,0)[lb]{\smash{{\SetFigFont{10}{12.0}{\rmdefault}{\mddefault}{\updefault}{\color[rgb]{0,0,0}$3$}%
}}}}
\put(-713,201){\makebox(0,0)[lb]{\smash{{\SetFigFont{10}{12.0}{\rmdefault}{\mddefault}{\updefault}{\color[rgb]{0,0,0}$2$}%
}}}}
\put(592,-186){\makebox(0,0)[lb]{\smash{{\SetFigFont{10}{12.0}{\rmdefault}{\mddefault}{\updefault}{\color[rgb]{0,0,0}$1$}%
}}}}
\put(1184,-482){\makebox(0,0)[lb]{\smash{{\SetFigFont{10}{12.0}{\rmdefault}{\mddefault}{\updefault}{\color[rgb]{0,0,0}$2$}%
}}}}
\put(2145,-482){\makebox(0,0)[lb]{\smash{{\SetFigFont{10}{12.0}{\rmdefault}{\mddefault}{\updefault}{\color[rgb]{0,0,0}$3$}%
}}}}
\put(2736,-186){\makebox(0,0)[lb]{\smash{{\SetFigFont{10}{12.0}{\rmdefault}{\mddefault}{\updefault}{\color[rgb]{0,0,0}$1$}%
}}}}
\end{picture}%

%% file: TorProjOdd.pdf_t
\begin{picture}(0,0)%
\includegraphics{TorProjOdd.pdf}%
\end{picture}%
\setlength{\unitlength}{4144sp}%
\begingroup\makeatletter\ifx\SetFigFont\undefined%
\gdef\SetFigFont#1#2#3#4#5{%
  \reset@font\fontsize{#1}{#2pt}%
  \fontfamily{#3}\fontseries{#4}\fontshape{#5}%
  \selectfont}%
\fi\endgroup%
\begin{picture}(3765,1168)(-149,-296)
\put(1396,254){\makebox(0,0)[lb]{\smash{{\SetFigFont{9}{10.8}{\rmdefault}{\mddefault}{\updefault}{\color[rgb]{0,0,0}$a$}%
}}}}
\put(-134,254){\makebox(0,0)[lb]{\smash{{\SetFigFont{9}{10.8}{\rmdefault}{\mddefault}{\updefault}{\color[rgb]{0,0,0}$a$}%
}}}}
\put(1036,749){\makebox(0,0)[lb]{\smash{{\SetFigFont{9}{10.8}{\rmdefault}{\mddefault}{\updefault}{\color[rgb]{0,0,0}$b$}%
}}}}
\put(226,-241){\makebox(0,0)[lb]{\smash{{\SetFigFont{9}{10.8}{\rmdefault}{\mddefault}{\updefault}{\color[rgb]{0,0,0}$b$}%
}}}}
\put(226,749){\makebox(0,0)[lb]{\smash{{\SetFigFont{9}{10.8}{\rmdefault}{\mddefault}{\updefault}{\color[rgb]{0,0,0}$c$}%
}}}}
\put(991,-241){\makebox(0,0)[lb]{\smash{{\SetFigFont{9}{10.8}{\rmdefault}{\mddefault}{\updefault}{\color[rgb]{0,0,0}$c$}%
}}}}
\put(3601,254){\makebox(0,0)[lb]{\smash{{\SetFigFont{9}{10.8}{\rmdefault}{\mddefault}{\updefault}{\color[rgb]{0,0,0}$a$}%
}}}}
\put(2431,-241){\makebox(0,0)[lb]{\smash{{\SetFigFont{9}{10.8}{\rmdefault}{\mddefault}{\updefault}{\color[rgb]{0,0,0}$b$}%
}}}}
\put(3241,-241){\makebox(0,0)[lb]{\smash{{\SetFigFont{9}{10.8}{\rmdefault}{\mddefault}{\updefault}{\color[rgb]{0,0,0}$c$}%
}}}}
\put(2431,749){\makebox(0,0)[lb]{\smash{{\SetFigFont{9}{10.8}{\rmdefault}{\mddefault}{\updefault}{\color[rgb]{0,0,0}$b$}%
}}}}
\put(3241,749){\makebox(0,0)[lb]{\smash{{\SetFigFont{9}{10.8}{\rmdefault}{\mddefault}{\updefault}{\color[rgb]{0,0,0}$c$}%
}}}}
\put(2071,254){\makebox(0,0)[lb]{\smash{{\SetFigFont{9}{10.8}{\rmdefault}{\mddefault}{\updefault}{\color[rgb]{0,0,0}$a$}%
}}}}
\end{picture}%

%% file: BrCov.pdf_t
\begin{picture}(0,0)%
\includegraphics{BrCov.pdf}%
\end{picture}%
\setlength{\unitlength}{4144sp}%
\begingroup\makeatletter\ifx\SetFigFont\undefined%
\gdef\SetFigFont#1#2#3#4#5{%
  \reset@font\fontsize{#1}{#2pt}%
  \fontfamily{#3}\fontseries{#4}\fontshape{#5}%
  \selectfont}%
\fi\endgroup%
\begin{picture}(3872,1585)(-28,-959)
\put(323,-904){\makebox(0,0)[lb]{\smash{{\SetFigFont{10}{12.0}{\rmdefault}{\mddefault}{\updefault}{\color[rgb]{0,0,0}$2$}%
}}}}
\put(1126,-904){\makebox(0,0)[lb]{\smash{{\SetFigFont{10}{12.0}{\rmdefault}{\mddefault}{\updefault}{\color[rgb]{0,0,0}$3$}%
}}}}
\put(1394, 33){\makebox(0,0)[lb]{\smash{{\SetFigFont{10}{12.0}{\rmdefault}{\mddefault}{\updefault}{\color[rgb]{0,0,0}$2$}%
}}}}
\put(724,503){\makebox(0,0)[lb]{\smash{{\SetFigFont{10}{12.0}{\rmdefault}{\mddefault}{\updefault}{\color[rgb]{0,0,0}$3$}%
}}}}
\put(-13, 33){\makebox(0,0)[lb]{\smash{{\SetFigFont{10}{12.0}{\rmdefault}{\mddefault}{\updefault}{\color[rgb]{0,0,0}$2$}%
}}}}
\put(2190,-132){\makebox(0,0)[lb]{\smash{{\SetFigFont{10}{12.0}{\rmdefault}{\mddefault}{\updefault}{\color[rgb]{0,0,0}$2$}%
}}}}
\put(2641,334){\makebox(0,0)[lb]{\smash{{\SetFigFont{10}{12.0}{\rmdefault}{\mddefault}{\updefault}{\color[rgb]{0,0,0}$3$}%
}}}}
\put(3489,253){\makebox(0,0)[lb]{\smash{{\SetFigFont{10}{12.0}{\rmdefault}{\mddefault}{\updefault}{\color[rgb]{0,0,0}$2$}%
}}}}
\put(3829,-281){\makebox(0,0)[lb]{\smash{{\SetFigFont{10}{12.0}{\rmdefault}{\mddefault}{\updefault}{\color[rgb]{0,0,0}$3$}%
}}}}
\put(3821,-538){\makebox(0,0)[lb]{\smash{{\SetFigFont{10}{12.0}{\rmdefault}{\mddefault}{\updefault}{\color[rgb]{0,0,0}$2$}%
}}}}
\put(3081,-830){\makebox(0,0)[lb]{\smash{{\SetFigFont{10}{12.0}{\rmdefault}{\mddefault}{\updefault}{\color[rgb]{0,0,0}$3$}%
}}}}
\put(698,-433){\makebox(0,0)[lb]{\smash{{\SetFigFont{10}{12.0}{\rmdefault}{\mddefault}{\updefault}{\color[rgb]{0,0,0}$1$}%
}}}}
\put(3085,-518){\makebox(0,0)[lb]{\smash{{\SetFigFont{10}{12.0}{\rmdefault}{\mddefault}{\updefault}{\color[rgb]{0,0,0}$2$}%
}}}}
\put(3058,-321){\makebox(0,0)[lb]{\smash{{\SetFigFont{10}{12.0}{\sfdefault}{\mddefault}{\updefault}{\color[rgb]{0,0,0}$1$}%
}}}}
\end{picture}%

%% file: ConjPoints.pdf_t
\begin{picture}(0,0)%
\includegraphics{ConjPoints.pdf}%
\end{picture}%
\setlength{\unitlength}{4144sp}%
\begingroup\makeatletter\ifx\SetFigFont\undefined%
\gdef\SetFigFont#1#2#3#4#5{%
  \reset@font\fontsize{#1}{#2pt}%
  \fontfamily{#3}\fontseries{#4}\fontshape{#5}%
  \selectfont}%
\fi\endgroup%
\begin{picture}(3991,1658)(-1361,-890)
\put(-282,-203){\makebox(0,0)[lb]{\smash{{\SetFigFont{8}{9.6}{\rmdefault}{\mddefault}{\updefault}{\color[rgb]{0,0,0}$\ell_k$}%
}}}}
\put(-1139,-588){\makebox(0,0)[lb]{\smash{{\SetFigFont{8}{9.6}{\rmdefault}{\mddefault}{\updefault}{\color[rgb]{0,0,0}$\frac{2\pi}k$}%
}}}}
\put(1801,209){\makebox(0,0)[lb]{\smash{{\SetFigFont{9}{10.8}{\sfdefault}{\mddefault}{\updefault}{\color[rgb]{0,0,0}$\ell_k$}%
}}}}
\put(1073,116){\makebox(0,0)[lb]{\smash{{\SetFigFont{9}{10.8}{\sfdefault}{\mddefault}{\updefault}{\color[rgb]{0,0,0}$\dev(a)$}%
}}}}
\put(1907,-340){\makebox(0,0)[lb]{\smash{{\SetFigFont{9}{10.8}{\sfdefault}{\mddefault}{\updefault}{\color[rgb]{0,0,0}$\dev(b)$}%
}}}}
\put(1627,-102){\makebox(0,0)[lb]{\smash{{\SetFigFont{9}{10.8}{\sfdefault}{\mddefault}{\updefault}{\color[rgb]{0,0,0}$\ell_k$}%
}}}}
\end{picture}%